\documentclass[11pt,a4paper]{article}

\usepackage{epsf,epsfig,amsfonts,amsgen,amsmath,amstext,amsbsy,amsopn,amsthm
}
\usepackage{ebezier,eepic}
\usepackage{color}

\newtheorem{theorem}{Theorem}

\newtheorem{lemma}{Lemma}
\newtheorem{false statement}{False statement}

\theoremstyle{definition}

\newtheorem{claim}{Claim}

\newtheorem{remark}[claim]{Remark}

\newcounter{mathitem}
  {\begin{list}{{$(\roman{mathitem})$}}{
   \setcounter{mathitem}{0}
   \usecounter{mathitem}
   \setlength{\topsep}{0pt plus 2pt minus 0pt}
   \setlength{\parskip}{0pt plus 2pt minus 0pt}
   \setlength{\partopsep}{0pt plus 2pt minus 0pt}
   \setlength{\parsep}{0pt plus 2pt minus 0pt}
   \setlength{\leftmargin}{35pt}
   \setlength{\itemsep}{0pt plus 2pt minus 0pt}}}
  {\end{list}}

\date{\dateline{April 12, 2013}{}}
\begin{document}

\title{\bf\Large On some papers of Nikiforov\thanks{supported by NSFC (No.~11271300) and the Doctorate Foundation of Northwestern Polytechnical University (cx201326).}}

\date{}

\author{Bo Ning\thanks{E-mail address: ningbo\_math84@mail.nwpu.edu.cn (B. Ning).}\\
\small Department of Applied Mathematics, School of Science,\\
\small  Northwestern Polytechnical University,\\
\small Xi'an, Shaanxi 710072, P.R.~China}
\maketitle
\begin{abstract}
The well known Mantel's Theorem states that a graph on $n$ vertices and $m$ edges contains a triangle if $m>\frac{n^2}{4}$. Nosal proved that every graph on $m$ edges contains a triangle if the spectral radius $\lambda_1>\sqrt{m}$, which is a spectral analog of Mantel's Theorem. Furthermore, by using Motzkin-Straus Inequality, Nikiforov sharped Nosal's result and characterized the extremal graphs when the equality holds. Our first contribution in this note is to give two new proofs of the spectral concise Mantel's Theorem due to Nikiforov (without help of Motzkin-Straus Inequality). Nikiforov also obtained some results concerning the existence of consecutive cycles and spectral radius. Second, we prove a theorem concerning the existence of consecutive even cycles and spectral radius, which slightly improves a result of Nikiforov. At last, we focus on spectral radius inequalities. Hong proved his famous bound for spectral radius. Later, Hong, Shu and Fang generalized Hong's bound to connected graphs with given minimum degree. By using quite different technique, Nikiforov proved Hong et al.'s bound for general graphs independently. In this note, we prove a new spectral inequality by applying the technique of Nikiforov. Our result extends Stanley's spectral inequality.


\end{abstract}
\medskip
\section{Introduction}
Throughout this note, we only consider graphs which are simple, undirected and finite. We refer the reader to \cite{BM} for terminology and notation not defined here.

Let $G$ be a graph on $n$ vertices, $V(G)$ and $E(G)$ denote the \emph{vertex set} and  \emph{edge set} of $G$, respectively. For a vertex $v\in V(G)$, the neighbor set of $v$ in $G$, denoted $N(v)$, is the set of vertices which are neighbors of $v$ in $G$. We denote $N[v]=N(v)\cup \{v\}$. Let $\delta$ and $d$ be the \emph{minimum degree} and \emph{average degree} of $G$, respectively. The \emph{adjacency matrix} of $G$, denoted by $A(G)$ (or $A$ for simple), is a matrix $A$ such that the $ij$-entry $a_{ij}=1$ if $v_i$ is adjacent to $v_j$ and $a_{ij}=0$ otherwise. Since $A$ is a symmetric $(0,1)$-matrix, all the eigenvalues of $A$ are real. The \emph{eigenvalues} of a graph $G$ are all the eigenvalues of its adjacency matrix. Let $\lambda_1\geq \lambda_2\geq\ldots \geq\lambda_n$ be all the eigenvalues of $G$. In the following, $\lambda_1$ is always called the \emph{spectral radius} of $G$.

The main results of this note are closely related to three past papers due to Nikiforov \cite{Ni1,Ni2,Ni3}. We first give two new proofs of a previous theorem on the existence of triangles and spectral radius of graphs due to Nikiforov. Second, we give a new result concerning consecutive even cycles and spectral radius, which slightly improves a result of Nikiforov. At last, by applying a previous technique of Nikiforov, we give a new spectral radius inequality which extends Stanley's spectral inequality.

\section{Spectral radius}
\subsection{Triangles and spectral radius}
In extremal graph theory, Mantel's Theorem may be a fundamental one. It is a starting point of the famous Tur\'{a}n's Theorem and has many structural proofs. We refer the reader to \cite{C}, where three brief and beautiful ones can be found.

\begin{theorem}
Let $G$ be a graph on $n$ vertices and $m$ edges. If $m>\frac{n^2}{4}$, then $G$ contains a triangle.
\end{theorem}

In 1970, Nosal \cite{No} proved an analogue of Mantel's Theorem in spectral graph theory.

\begin{theorem}[Nosal \cite{No}]
Let $G$ be a graph on $m$ edges and $\lambda_1$ be the spectral radius of $G$. If $\lambda_1>\sqrt{m}$, then $G$ contains a triangle.
\end{theorem}

Proving a conjecture of Edwards and Elphick \cite{EE}, Nikiforov \cite{Ni1} obtained a spectral concise Tur\'{a}n's Theorem which generalizes Nosal's theorem \cite{No}. One important technique in Nikiforov's proof is to use Motzkin-Straus Inequality \cite{MS}, and its power and close connection with extremal problems in spectral graph theory was firstly noticed by Wilf \cite{W}. In \cite{Ni3}, Nikiforov characterized all the extremal cases.

We list the part (i) of \cite[Theorem 2]{Ni3} as follows.

\begin{theorem}[Nikiforov \cite{Ni3}]
Let $G$ be a graph on $m$ edges and $\lambda_1$ be the spectral radius of $G$. If $\lambda_1\geq \sqrt{m}$, then $G$ contains a triangle unless $G$ is a complete bipartite graph with possible some isolated vertices.
\end{theorem}

Since $\lambda_1$ is at least the average degree of $G$, one can deduce Mantel's Theorem from Theorem 3. Furthermore, to prove Theorem 3, we only need prove the following theorem.

\begin{theorem}[Nikiforov \cite{Ni3}]
Let $G$ be a connected graph on $m$ edges and $\lambda_1$ be the spectral radius of $G$. If $\lambda_1\geq \sqrt{m}$, then $G$ contains a triangle unless $G$ is a complete bipartite graph.
\end{theorem}

The direct motivation of this subsection is to give two new proofs of Theorem 4 (without help of Motzkin-Straus Inequality).

To give the first proof, the following two lemmas are needed.

\begin{lemma}\cite[p.38]{BH}
A graph $G$ is bipartite if and only if for each eigenvalue $\lambda$ of $G$, $-\lambda$ is also an eigenvalue, with the same multiplicity.
\end{lemma}

\begin{lemma}\cite[p.5]{BH}
Let $G$ be a connected graph with diameter $d$. Then $G$ has at least $d+1$ distinct eigenvalues.
\end{lemma}

\noindent{}
{\bf The first new proof of Theorem 4.}

Let $\lambda_1\geq\lambda_2\geq\ldots \lambda_n$ be the eigenvalues of $G$. By the well known fact (see \cite[p.85]{CDS})
\begin{align*}
2m={\lambda_1}^2+{\lambda_2}^2+\ldots+{\lambda_n}^2
\end{align*}
and the inial condition
\begin{align*}
\lambda_1\geq \sqrt{m},
\end{align*}
we can deduce
\begin{align}
{\lambda_1}^2\geq {\lambda_2}^2+\ldots +{\lambda_n}^2.
\end{align}
Let $t(G)$ be the number of triangles in $G$. It is known (see \cite[p.85]{CDS})
\begin{align}
t(G)=\frac{{\lambda_1}^3+{\lambda_2}^3+\ldots +{\lambda_n}^3}{6}.
\end{align}
By putting inequality (1) into (2), we can obtain
\begin{align*}
t(G)&\geq\frac{{\lambda_1}({\lambda_2}^2+\ldots +{\lambda_n}^2)+{\lambda_2}^3+\ldots +{\lambda_n}^3}{6}\\
&=\frac{{\lambda_2}^2({\lambda_1}+{\lambda_2})+{\lambda_3}^2({\lambda_1}+{\lambda_3})+\ldots+{\lambda_n}^2({\lambda_1}+{\lambda_n})}{6}.
\end{align*}
Note that ${\lambda_1}+{\lambda_2}+{\lambda_3}+\ldots+{\lambda_n}=0$. By the famous Perron-Frobenius theorem, $|\lambda_i|\leq \lambda_1$ for $i=2,3,\ldots,n$ (since $G$ is connected).
Thus $t(G)=0$ if and only if
\begin{align}
{\lambda_i}^2({\lambda_1}+{\lambda_i})=0
\end{align}
for $i=2,3,\ldots,n$.

Now assume that  ${\lambda_i}^2({\lambda_1}+{\lambda_i})=0$ for $i=2,3,\ldots,n$. If $\lambda_i=0$ for all $i=2,3,\ldots,n$, then by the
trace condition on the adjacency matrix, $\lambda_1=0$, and this implies that $G$ is empty, a contradiction. Thus there exists an integer $j\in \{2,3,\ldots,n\}$, such that $\lambda_j\neq 0$, and this implies  $\lambda_j=-\lambda_1$. By Lemma 1, $G$ is bipartite. Since $G$ has only three distinct eigenvalues when $t(G)=0$, by Lemma 2, the diameter of $G$ is at most two. However, if $G$ is not a complete bipartite graph, then the diameter of $G$ is at least three, a contradiction. Thus $G$ is a complete bipartite graph.

The proof is complete.{\hfill$\Box$}

To give the second proof of Theorem 4, the following lemma is needed.

\begin{lemma}\cite[p.203]{FMS}
Let $G$ be a graph with the vertex set $V$. Then $$\lambda_1\leq \sqrt{\max\{\sum_{u\in N(v)}d(u):v\in V\}}.$$ If $G$ is connected, then the equality holds if and only if $G$ is regular or bipartite semi-regular.
\end{lemma}

\noindent{}
{\bf The second new proof of Theorem 4.}

Assume that $G$ contains no triangle. Thus, for any vertex $v\in V(G)$, $\sum_{u\in N(v)}d(u)= e(N(v),V(G)\backslash N(v))$. This implies that
\begin{align}
\sum_{u\in N(v)}d(u)\leq m.
\end{align}
Suppose that the vertex $v_0$ satiesfies $\sum_{u\in N(v_0)}d(u)=\max\{\sum_{u\in N(v)}d(u):v\in V\}$. By Lemma 3 and the inequality (4), we have
\begin{align}
\sqrt{\sum_{u\in N(v_{0})}d(u)} \leq \sqrt{m}\leq \lambda_1\leq \sqrt{\sum_{u\in N(v_{0})}d(u)}.
\end{align}
Thus all inequalities in (5) become equalities. By Lemma 3, $G$ is regular or bipartite semi-regular. If $G$ is $k$-regular, then $\lambda_1=k=\sqrt{\frac{kn}{2}}$. It follows that $k=n/2$. Since $G$ contains no triangles, $G=K_{n/2,n/2}$. If $G$ is bipartite semi-regular, then
assume that $G=G[A,B]$, where $A$ and $B$ are two parts with $|A|=a$ and $|B|=b$. Assume that the degree of every vertex in $A$ is $r$ and the degree of every vertex in $B$ is $s$. Obviously, $r\leq b$ and $s\leq a$. W.l.o.g., assume that $v_0\in A$. Then $\sum_{u\in N(v_0)}d(u)=rs$, and this implies that $m=rs$ by (5). On the other hand, we have $m=ra=sb$, and it follows that $r=b$ and $s=a$. Thus $G$ is a complete bipartite graph.

The proof is complete. {\hfill$\Box$}

\subsection{Consecutive even cycles and spectral radius}
Nikiforov \cite{Ni2} proved the following theorem, which contributes to spectral extremal graph theory.

\begin{theorem}[Nikiforov \cite{Ni2}]
Let $G$ be a graph of sufficiently large order $n$ with $\lambda_1>\sqrt{\lfloor n^2/4\rfloor}$. Then $C_l\subset G$ for every $3\leq l\leq n/320$.
\end{theorem}

The following is a direct corollary.

\begin{theorem}[Nikiforov \cite{Ni2}]
Let $G$ be a graph of sufficiently large order $n$ with $\lambda_1>\sqrt{\lfloor n^2/4\rfloor}$. Then $C_l\subset G$ for every even $4\leq l\leq n/320$.
\end{theorem}

By using some spectral inequality and results from structural graph theory, we prove a theorem which slightly improves Theorem 6.
\begin{theorem}
Let $G$ be a graph of sufficiently large order $n$ with $\lambda_1>\sqrt{\lfloor n^2/4\rfloor}$. Then $C_l\subset G$ for every even $4\leq l\leq \lceil \frac{n}{28}\rceil$.
\end{theorem}

To give the proof, the following four lemmas are needed.
\begin{lemma}[Stanley \cite{S}]
Let $G$ be a graph on $n$ vertices and $m$ edges. Then $\lambda_1\leq-\frac{1}{2}+\sqrt{2m+\frac{1}{4}}$.
\end{lemma}

\begin{lemma}[Erd\"{o}s \cite{E}]
Let $G$ be a graph and $d$ be its average degree. If $d\geq 2k$, where $k$ is a positive integer, then $G$ contains an induced subgraph with minimum degree at least $k+1$.
\end{lemma}

\begin{lemma}[Bondy \cite{B}]
Let $G$ be a graph on $n$ vertices and $\delta> n/2$ be the minimum degree. Then $C_l\subset G$ for $3\leq l\leq n$.
\end{lemma}

\begin{lemma}[Allen et al. \cite{ABHC}]
Let $G$ be a graph on $n$ vertices and $\delta\geq n/k$ be the minimum degree, where $k$ be an integer. If $n\geq n_0=O(k^{20})$, then $C_l\subset G$ for every even $4\leq l\leq \lceil \frac{n}{k-1}\rceil$.
\end{lemma}

\noindent{}
{\bf Proof of Theorem 7.}
By Lemma 4, $\lambda_1\leq-\frac{1}{2}+\sqrt{2m+\frac{1}{4}}$. By the inial condition of Theorem 7, $\lambda_1>\sqrt{\lfloor n^2/4\rfloor}$. Hence $2m\geq n^2/4+\sqrt{\lfloor\frac{n^2}{4}\rfloor}-1\geq \frac{n^2}{4}$ and $d(G)=\frac{2m}{n}>\frac{n}{4}$. By Lemma 5, there is an induced subgraph $H$ of $G$, such that $\delta(H)\geq \frac{n}{8}$. Let $n'$ be the order of $H$. If $n'\leq\frac{n}{4}$, then by Lemma 6, $H$ contains cycles of lengths $3$ to $\frac{n}{4}$, and there is nothing to do. Now assume that $n'>\frac{n}{4}$, and we have $\delta(H)\geq \frac{n'}{8}$.
By setting $k=8$ in Lemma 7, we know that $H$ contains $C_t$ for every even $4\leq l\leq \lceil \frac{n'}{7}\rceil$. It follows that $H$ contains $C_l$ for every even $4\leq l\leq \lceil \frac{n}{28}\rceil$, and thus $G$ contains the corresponding cycles. The proof is complete. {\hfill$\Box$}

\subsection{An extension of a spectral inequality of Stanley}
The bounds of spectral radius of graphs have received much attentions from spectral graph theorists. The following is Hong's famous bound.

\begin{theorem}[Hong \cite{H}]
Let $G$ be a graph on $m$ edges and $n$ vertices, and without isolated vertices. Then $\lambda_1\leq \sqrt{2m-n+1}$.
\end{theorem}

Hong, Shu and Fang \cite{HSK} generalized Theorem 8 to a connected graph with a given minimum degree. They also characterize the extremal graphs when the equality holds.

\begin{theorem}[Hong, Shu and Fang \cite{HSK}]
Let $G$ be a connected graph on $n$ vertices and $m$ edges and let $\delta$ be the minimum degree of $G$. Then $\lambda_1\leq\frac{\delta-1}{2}+\sqrt{2m-n\delta+\frac{(\delta+1)^2}{4}}$. Equality holds if and only if $G$ is either a regular graph or a bidegreed graph in which each vertex is of degree either $\delta$ or $n-1$.
\end{theorem}

Nikiforv \cite{Ni1} proved the above inequality for a general graph independently by a quite different method.

\begin{theorem}[Nikiforov \cite{Ni1}]
Let $G$ be a graph on $n$ vertices and $m$ edges and let $\delta$ be the minimum degree of $G$. Then $\lambda_1\leq\frac{\delta-1}{2}+\sqrt{2m-n\delta+\frac{(\delta+1)^2}{4}}$.
\end{theorem}

Nikiforov's proof technique is based on the inequality of the number of walks of graphs. Furthermore, Nikiforov found that there are more general graph classes when the equality holds in Theorem 10, as shown in \cite{Ni1,ZC}. On the other hand, Hong et al.'s result can be used to characterize all the extremal graphs when the equality holds in Theorem 10 (for general graphs), as shown by Zhou and Cho \cite{ZC} .

By applying Nikiforov's technique (See the proof of Theorem 4.1 in \cite{Ni1}), we extend a spectral inequality of Stanley (See Lemma 4 in Subsec. 2.2).
\begin{theorem}
Let $G$ be a graph with the vertex set $V$, where $|V|=n$. Then
\begin{align}
\lambda_1\leq \frac{1}{2}(-1+\sqrt{1+4\max\{\sum_{u\in N[v]}d(u):v\in V}\}).
\end{align}
\end{theorem}
\begin{proof}
The technique used here is originally introduced by Nikiforov \cite{Ni1}.
Let $w_k$ be the number of walks of length $k$ in $G$. Let $w_k(i)$ be the number of walks of length $k$ starting at the vertex $v_i$.
Note that the number of walks of length $k$ starting at $i$ and ending at $i$ are the same. Then
\begin{align*}
w_k&=\sum_{v_i\in V} w_{k-2}(i)w_2(i)\\
&=\sum_{v_i\in V} w_{k-2}(i)\sum_{j\in N(v_i)}d(v_j)\\
&=\sum_{v_i\in V} w_{k-2}(i)(\sum_{j\in N[v_i]}d(v_j)-d(v_i))
\end{align*}
\begin{align*}
&\leq\sum_{v_i\in V} w_{k-2}(i)\max\{\sum_{u\in N[v]}d(u):v\in V\}-w_{k-1}\\
&=\max\{\sum_{u\in N[v]}d(u):v\in V\}\cdot w_{k-2}-w_{k-1}.
\end{align*}
This implies that
\begin{align}
\frac{w_k}{w_{k-2}}+\frac{w_{k-1}}{w_{k-2}}-\max\{\sum_{u\in N[v]}d(u):v\in V\}\leq 0.
\end{align}
Let $\lambda_1\geq\cdots \geq\lambda_n$ be all the eigenvalues of $G$. Recall that the expression of $w_k$ is given by
$$
w_k=\sum_{i=1}^{n}c_i\lambda_i^k,
$$
where $c_i$'s are nonnegative real numbers (see \cite[p. 44, Theorem 1.10]{CDS}).
If $|\lambda_n|<\lambda_1$, then the modulus of any negative eigenvalue of $G$ is less than $\lambda_1$.
By simple computation, we have $\lim\limits_{k\longrightarrow\infty}\frac{w_k}{w_{k-2}}=\lambda_1^2$ and $\lim\limits_{k\longrightarrow\infty}\frac{w_{k-1}}{w_{k-2}}=\lambda_1$. Let $k$ tend infinitely in inequality (7), we have the inequality
\begin{align}
{\lambda_1}^2+\lambda_1-\max\{\sum_{u\in N[v]}d(u):v\in V\}\leq 0.
\end{align}
Now assume that $\lambda_n=-\lambda_1$. W.l.o.g., assume that there are some integers $r,s$ such that
$\lambda_1=\ldots=\lambda_r\geq\lambda_{r+1}\geq\cdots \lambda_{s-1}\geq\lambda_s\cdots=\lambda_n=-\lambda_1$. Set
$a=\sum_{i=1}^r c_i$ and $b=\sum_{i=s}^n c_i$. Now we show that $a>b$. (Compared with the clue given in the proof of Theorem 4.1 in \cite{Ni1}, we will include the reason why $a\neq b$ here.) If $a<b$, then noting that $\lim\limits_{k\longrightarrow\infty}\frac{w_{2k+1}}{{\lambda_1}^{2k+1}}=a-b<0$, we know that $w_{2k+1}$ will be negative if $k$ is sufficiently large, a contradiction. If $a=b$, then
$w_{2k}=2a{\lambda_1}^{2k}+\sum_{i=r+1}^{s-1}c_i{\lambda_i}^{2k}$ and $w_{2k+1}=\sum_{i=r+1}^{s-1}c_i{\lambda_i}^{2k+1}$. Thus we have $\lim\limits_{k\longrightarrow\infty}\frac{w_{2k+1}}{w_{2k}}=0$, a contradiction to the easy observation that $w_{2k+1}\geq w_{2k}$. Hence we deduce that $a>b$. It is easy to check that $\lim\limits_{k\longrightarrow\infty}\frac{w_{k}}{w_{k-2}}={\lambda_1}^2$, and $\lim\limits_{k\longrightarrow\infty}\frac{w_{2k}}{w_{2k-1}}=\frac{a+b}{a-b}\lambda_1\geq\lambda_1$. Replacing the subscript $k$ in inequality (7) by $2k+1$ and letting $k$ tend infinitely, we obtain the inequality (8).
Solving the equation, we obtain the inequality (6). The proof is complete.
\end{proof}

\begin{remark}
Since $\max\{\sum_{u\in N[v]}d(u):v\in V\}\leq 2m$, Theorem 11 can imply Stanley's spectral inequality.
\end{remark}

\noindent{}
{\bf Acknowledgements.} The author is particularly grateful to the anonymous referee for his/her careful reading the manuscript
and many invaluable suggestions.

\end{document}